\documentclass[12pt]{article}

\usepackage{amsmath,amssymb,bm,amsthm,mathrsfs,amscd}
\newtheorem{Def}{Def}[section]
\newtheorem{Defi}[Def]{Definition}
\newtheorem{Them}[Def]{Theorem}
\newtheorem{Lem}[Def]{Lemma}
\newtheorem{Cor}[Def]{Corollary}
\newtheorem{Prop}[Def]{Proposition}

\numberwithin{equation}{section}

\newcommand{\Hom}{\operatorname{Hom}}

\newcommand{\GL}{\operatorname{GL}}

\newcommand{\NB}{\mathbb{N}}

\newcommand{\RB}{\mathbb{R}}

\title
{A canonical system of basic invariants of a finite reflection group}
\author{Norihiro Nakashima\footnote{email: naka\_n@math.sci.hokudai.ac.jp}
 and Shuhei Tsujie\footnote{email: tsujie@math.sci.hokudai.ac.jp}}
\date{}

\begin{document}

\maketitle

\begin{abstract}
A canonical system of basic invariants is a system of invariants 
satisfying a set of differential equations. 
The properties of a canonical 
system are related to the mean value property for polytopes. 
In this article, we naturally identify the vector space spanned 
by a canonical system of basic invariants with an invariant 
space determined by a fundamental antiinvariant. 
From this identification, we obtain explicit formulas of 
canonical systems of basic invariants. 
The construction of the formulas does not depend on the 
classification of finite irreducible reflection groups. 
\vspace{5mm}
\vspace{2mm}
\end{abstract}

\section{Introduction}
Let $V$ be a real $n$-dimensional Euclidean space, 
and $W\subseteq O(V)$ a finite reflection group. 
Let $S$ denote the symmetric algebra $S(V^{\ast})$ of the 
dual space $V^{\ast}$, and $S_{k}$ the vector space 
consisting of homogeneous polynomials of degree $k$ and the zero polynomial. 
Then $W$ naturally acts on $S$ and $S_{k}$. 
According to Chevalley \cite{Che}, the subalgebra $R=S^{W}$ 
of $W$-invariant polynomials of $S$ is generated by 
$n$ algebraically independent homogeneous polynomials. 
A system of such generators is called a system of basic invariants of $R$. 
It is easy to construct systems of basic invariants 
for reflection groups of the types $A_{n}$, $B_{n}$, $D_{n}$ and $I_{2}$. 
Many researchers constructed explicit systems of basic invariants 
for a reflection group of each type 
(Coxeter \cite{Cox}, Mehta \cite{Mehta}, 
Saito, Yano, and Sekiguchi \cite{SYS}, and Sekiguchi and Yano \cite{SY1,SY2}).

Let $v_{1},\dots,v_{n}$ be an orthonormal basis for $V$, 
$x_{1},\dots,x_{n}$ the basis for $V^{\ast}$ 
dual to $v_{1},\dots,v_{n}$, 
and $\partial_{1},\dots,\partial_{n}$ the basis for 
$V^{\ast\ast}$ dual to $x_{1},\dots,x_{n}$. 
Although we may identify $V^{\ast\ast}$ with $V$ naturally, 
we usually distinguish the basis 
$\partial_{1},\dots,\partial_{n}$ from the basis $v_{1},\dots,v_{n}$. 
We define a bilinear map $\left(\cdot,\cdot \right):S\times S\rightarrow S$ by 
\begin{align}
\left(f,g \right)=f(\partial)g(x) \qquad (f,g\in S),\label{bil}
\end{align}
where $x=(x_{1},\dots,x_{n})$ and 
$\partial=(\partial_{1},\dots,\partial_{n})$. 
Define an inner product 
$\langle\cdot,\cdot\rangle:S\times S\rightarrow \RB$ by 
\begin{align}
\langle f,g \rangle = f(\partial)g(x)|_{x=0}\qquad (f,g\in S).\label{inner}
\end{align}
It is not hard to see that $ \langle f,g \rangle = \langle g,f \rangle $ for $ f,g \in S $.

Two systems $g_{1},\dots,g_{n}$ and $h_{1},\dots,h_{n}$ 
of basic invariants are said to be equivalent if 
there exists $A\in\GL_{n}(\RB)$ such that 
$$
[h_{1},\dots,h_{n}]=[g_{1},\dots,g_{n}]A.
$$
\begin{Defi}\label{cano-def}
A system $f_{1},\dots,f_{n}$ of basic invariants 
is said to be {\bf canonical} if it satisfies the following 
system of partial differential equations: 
\begin{align}
\left(f_{i},f_{j} \right)=
\begin{cases}
1&{\rm if}\ i=j\\
0&{\rm otherwise}
\end{cases}\label{def-cbi}
\end{align}
for $i,j=1,\dots,n$. 
\end{Defi}

The canonical system was first defined by 
Flatto \cite{Fla1,Fla2} and Flatto and Wiener \cite{Fla-Wie} 
for determining the structure of the linear space 
consisting of $P(0)$-harmonic functions, where 
$P(k)$ is the $k$-skeleton of an $n$-dimensional polytope $P$ 
for a nonnegative integer $0\leq k\leq n-1$. 
(An $\RB$-valued continuous function is called 
a $P(k)$-harmonic function if it satisfies 
the mean value property on $P(k)$.)
Flatto \cite{Fla1,Fla2} and Flatto and Wiener \cite{Fla-Wie} 
verified that there exists a unique (up to equivalence) 
canonical system of basic invariants for any finite reflection group, 
and gave an algorithm to find the canonical system. 
In \cite{Fla1,Fla2,Fla-Wie}, a canonical system 
of basic invariants was found as a solution of a certain 
system of partial differential equations. 
Definition \ref{cano-def} is due to Iwasaki \cite{Iwa}.

Explicit formulas for canonical systems play an 
important role when we determine the structure of the linear space 
consisting of $P(k)$-harmonic functions. Especially, 
from the argument using explicit formulas of the canonical systems, 
the structure of the linear space consisting of $P(k)$-harmonic functions 
was determined for any $0\leq k\leq n-1$ 
when $P$ is a regular convex polytope (see \cite{Iwa2,Iwa3,IKM}). 
(When $P$ is a regular convex polytope, the symmetry group 
of $P$ is a finite reflection group of the type $A$, $B$, $F$, $H$, or $I$.) 

However, the algorithm for constructing a canonical system 
(given by Flatto \cite{Fla1,Fla2} and Flatto and Wiener \cite{Fla-Wie}) 
does not seem to be effective in practice. 
(It is hard to give an explicit formula of a canonical system 
from the algorithm.) 
It is an interesting problem to determine canonical systems. 
Iwasaki \cite{Iwa} gave explicit formulas of canonical systems 
for reflection groups of the types $A_{n}$, $B_{n}$, 
$D_{n}$ and $I_{2}$. Iwasaki, Kenma and Matsumoto \cite{IKM} gave 
explicit formulas of canonical systems for 
the irreducible finite reflection groups 
of the types $F_{4}$, $H_{3}$ and $H_{4}$. 
The problem for determining canonical systems of basic invariants 
for the remaining types ($E_{6}$, $E_{7}$ and $E_{8}$) has been open. 
In this article, we explicitly construct canonical systems of 
basic invariants from an arbitrary system of basic invariants
(Theorem \ref{const-can-sys}). 
The construction does not depend on the classification for 
finite irreducible reflection groups. 

In contrast, making use of the classification, we may refine our construction.
In the case of the types other than $D_{n}$ (with even $n$), 
we may have a straightforward construction
(Theorem \ref{can-sys1}). 
We need to consider the case of the type $D_{n}$ with even $n$ separately. 
Iwasaki \cite{Iwa} constructed a canonical system of 
basic invariants containing the monomial 
$\prod_{i=1}^n x_{i}=x_{1}\cdots x_{n}$. 
By using the monomial $\prod_{i=1}^n x_{i}$, 
we obtain a construction for a canonical system arising from an arbitrary 
system of basic invariants (Theorem \ref{can-sys2}). 

Let $\Phi$ be the root system associated with a finite reflection group $W$, 
and $\Phi^{+}$ a positive system in the sense of (5.4) in \cite{Hum}. 
For $\alpha\in\Phi$, fix a homogeneous polynomial $L_{\alpha}$ 
of degree $1$ defining the reflecting hyperplane $H_{\alpha}$ 
(i.e., $\ker L_{\alpha}=H_{\alpha}$). 
A polynomial $f\in S$ is said to be an antiinvariant if 
$wf=(\det w)f$ for all $w\in W$. 
Put $\Delta=\prod_{\alpha\in\Phi^{+}}L_{\alpha}$, then 
$\Delta$ is an antiinvariant. Set 
$$
\RB[\partial]\Delta:=\left\{f(\partial)\Delta\mid f\in S\right\}.
$$
We naturally identify the vector space spanned by a canonical system 
of basic invariants with $\Omega_{W} := (\RB[\partial]\Delta\otimes_{\RB} V^{\ast})^W$. 
It is known that the graded vector space $\RB[\partial]\Delta$ 
affords the regular representation of $W$ 
(see Bourbaki \cite{Bou} and Steinberg \cite{Ste}). 
This is a key to constructing canonical systems. 
\section{Characterization of the canonical systems}

Let $R_{+}$ be the ideal of $R$ generated by homogeneous elements of 
positive degrees, and $I=SR_{+}$ the ideal of $S$ generated by $R_{+}$. 
The following key lemma is obtained by Steinberg \cite{Ste}. 
\begin{Lem}\label{Stein1}
Let $f\in S$ be a homogeneous polynomial. Then we have the following:
\begin{enumerate}
\item[$(1)$] $f\in I$ if and only if $f(\partial)\Delta =0$,
\item[$(2)$] $g(\partial)f=0$ for any $g\in I$ if and only if 
$f\in\RB[\partial]\Delta$, 
\item[$(3)$] $I$ is the orthogonal complement of 
$\RB[\partial]\Delta$, and $S=I\oplus\RB[\partial]\Delta$. 
\end{enumerate}
\end{Lem}
It is known that a $W$-stable graded subspace $U$ of $S$ such that 
$S=I\oplus U$ is isomorphic to the regular representation 
(see Bourbaki \cite[Chap. 5 Sect. 5 Theorem 2]{Bou}). 
By Lemma \ref{Stein1}, the graded vector space $\RB[\partial]\Delta$ 
affords the regular representation of $W$. 

In the rest of this paper, we assume that 
$W$ is irreducible and $V$ is generated by the roots. 
Then any endomorphism of $V$ is a multiplicative map 
with a constant in $\RB$, and $V$ is absolutely irreducible 
(see Bourbaki \cite[Chap. 5, Sect. 2, Proposition 1]{Bou} 
and \cite[Chap. 5, Sect. 3, Proposition 5]{Bou}). 
Therefore, by the general theory of group representations, we have 
that the multiplicity of $V$ in the regular representation 
is $\dim_{\RB} V=n$. Note that there exists a $W$-module $U^{\prime}$ 
such that $\RB[\partial]\Delta=V^{n}\oplus U^{\prime}$. 

We set
\begin{align*}
\Omega_{W} := \left( \mathbb{R}[\partial]\Delta \otimes_{\mathbb{R}}V^{\ast}\right)^{W}. 
\end{align*}
Then, according to Orlik and Solomon \cite{Or-Solo}, 
we have the isomorphism 
\begin{align}
\Omega_{W} = (\RB[\partial]\Delta\otimes_{\RB} V^{\ast})^W
\simeq\Hom_{W}(V,\RB[\partial]\Delta)\label{Hom-Tensor}
\end{align}
as $W$-modules. Since $V$ is absolutely irreducible, 
we have $ \dim_{\mathbb{R}} \Omega_{W} = n $ by \cite{Or-Solo} (pp. 80). 

There exists a unique $ \mathbb{R} $-linear map 
$d\,:\,S\longrightarrow S\otimes_{\RB}V^{\ast}$ satisfying 
$d(fg)=f(dg)+g(df)$ for $f,g\in S$ and $dL := 1 \otimes L \in\RB\otimes_{\RB}V^{\ast}$ 
for $L\in V^{\ast}$. The map $d$ is called the differential map. 
We set $\Omega^1(V):=S\otimes_{\RB}V^{\ast}$, then $\Omega^1(V)$ 
naturally has an $S$-graded structure. 
The differential map was associated with basic invariants 
by Solomon \cite{Sol1}. 
For $h\in S$, the differential $dh$ is given by the 
following formula 
$$
dh=\sum_{j=1}^n \partial_{j}h\otimes x_{j}=
\sum_{j=1}^n (\partial_{j}h) dx_{j}.
$$
Hence $dh$ is invariant if $h$ is invariant. 
Assume that $ df = 0 $ for $ f \in R_{+}$. 
Then $ f $ is a constant, and $ f = 0 $. 
Hence $ d|_{R_{+}} $ is injective. 

Recall that there always exists a homogeneous canonical system 
$\{f_{1},\dots,f_{n}\}$ by Flatto \cite{Fla1,Fla2} and 
Flatto and Wiener \cite{Fla-Wie}.
\begin{Lem}\label{basis-Omega}
For a homogeneous canonical system $\{f_{1},\dots,f_{n}\}$, 
the $1$-forms $df_{1},\dots,df_{n}$ are a basis for $\Omega_{W}$ 
over $\mathbb{R}$.
\end{Lem}
\begin{proof}
Let $\{i,j,k\}\subseteq \{1,\dots,n\}$. One has
$$
0=\partial_{k}(f_{j},f_{i})=\partial_{k}(f_{j}(\partial)f_{i})
=f_{j}(\partial)(\partial_{k}f_{i}).
$$
Thus $g(\partial)(\partial_{k}f_{i})=0$ for any $g\in R_{+}$. 
By Lemma \ref{Stein1} $(2)$, we have 
$\partial_{k}f_{i}\in\mathbb{R}[\partial]\Delta$ and $df_{i}\in\Omega_{W}$. 
Note that the $1$-forms $df_{1},\dots,df_{n}$ are 
linearly independent over $\mathbb{R}$. Since $\dim_{\mathbb{R}}\Omega_{W}=n$, 
the $1$-forms $df_{1},\dots,df_{n}$ are a basis for $\Omega_{W}$ 
over $\mathbb{R}$.
\end{proof}
Define the linear map 
$$
\varepsilon\,:\,(S\otimes_{\RB}V^{\ast})^W\longrightarrow R_{+}
$$
by 
$$
\varepsilon\left(\sum_{k=1}^n h_{k}dx_{k}\right)=\sum_{k=1}^n x_{k}h_{k}.
$$
Then $\varepsilon(dh)=(\deg h)h$ for any homogeneous polynomial $h\in R_{+}$. 
Define
\begin{align}
\mathcal{F}:=\varepsilon(\Omega_{W}).
\end{align}
By Lemma \ref{basis-Omega}, 
$\mathcal{F}=\langle f_{1},\dots,f_{n}\rangle_{\mathbb{R}}$ 
when $\{f_{1},\dots,f_{n}\}$ is a homogeneous canonical system. 
We thus have the following two $W$-isomorphisms:
\begin{align*}
d|_{\mathcal{F}}&:\mathcal{F}\overset{\sim}{\rightarrow}\Omega_{W},\\
\varepsilon|_{\Omega_{W}}&:\Omega_{W}
\overset{\sim}{\rightarrow}\mathcal{F}.
\end{align*}
In particular, an arbitrary element of $\Omega_{W}$ 
can be uniquely expressed as $dg$ for some $g\in\mathcal{F}$.

We now introduce an inner product
$$
(\cdot,\cdot):\Omega_{W}\times\Omega_{W}\rightarrow\mathbb{R}
$$
by
\begin{align}\label{Inner product on Omega_{W}}
(\omega_{1},\omega_{2})=\left(\sum_{j=1}^{n}g_{j}dx_{j},
\sum_{j=1}^{n}h_{j}dx_{j} \right)
:=\sum_{j=1}^{n}\langle g_{j},h_{j}\rangle, 
\end{align}
for $\omega_{1}=\sum_{j=1}^{n}g_{j}dx_{j},\ 
\omega_{2}= \sum_{j=1}^{n}h_{j}dx_{j}\in\Omega_{W}$. 
For homogeneous polynomials $f,g\in \mathcal{F}$, we have 
\begin{align}
(df,dg)&=\sum_{j=1}^{n}(\partial_{j}f,\partial_{j}g)
=\sum_{j=1}^{n}\langle x_{j}\partial_{j}f,g\rangle\notag \\ 
&=(\deg f)\langle f,g\rangle=
\begin{cases}
(\deg f)(f,g)\ &{\rm if}\ \deg f=\deg g,\\
0&{\rm otherwise}.
\end{cases}
\end{align}
Therefore two $W$-isomorphisms
\begin{align*}
d|_{\mathcal{F}}&:\mathcal{F}\overset{\sim}{\rightarrow}\Omega_{W},\\
\varepsilon|_{\Omega_{W}}&:\Omega_{W}
\overset{\sim}{\rightarrow}\mathcal{F}
\end{align*}
both preserve the orthogonality of homogeneous elements.
\begin{Cor}
Let $\{\omega_{1},\dots,\omega_{n}\}$ be an orthogonal basis 
consisting of homogeneous elements for $ \Omega_{W} $. 
Then the normalization of 
$ \varepsilon(\omega_{1}),\dots,\varepsilon(\omega_{n}) $ 
forms a canonical system. 
\end{Cor}
\section{Construction}\label{ex-form}

The following map is a key to our construction of the canonical system. 
\begin{Defi}
Define a map 
\begin{align}
\phi\,:\, S\longrightarrow\RB[\partial]\Delta,
\ \phi(f):=((f,\Delta),\Delta)\qquad
{\rm for\ }f\in S
\end{align}
where $(\cdot,\cdot)$ is the bilinear map \eqref{bil}.
\end{Defi}
\begin{Prop}\label{Prop:property of phi}
The map $ \phi $ satisfies the following:
\begin{enumerate}
\item[$(1)$] $ \phi $ is a $ W $-homomorphism, 
\item[$(2)$] $ \phi $ preserves a homogeneous component of $ S $, i.e., $ \phi(S_{k}) \subseteq S_{k} $ for any nonnegative integer $ k $, 
\item[$(3)$] $ \ker \phi = I $, 
\item[$(4)$] $ \phi $ is symmetric with respect to the inner product 
$\langle\cdot,\cdot\rangle$ \eqref{inner}, i.e.,
\begin{align*}
\langle \phi(f), g \rangle = \langle f, \phi(g) \rangle
\end{align*}  
for $ f, g \in S $. 
Therefore, $ \phi : S \rightarrow S $ is diagonalizable 
by homogeneous polynomials. 
\end{enumerate}
\end{Prop}
\begin{proof}
(1) For any $ w \in W$, we have
\begin{align*}
w \cdot ((f,\Delta),\Delta) = ((w \cdot f, \det w \Delta), \det w \Delta) = ((w\cdot f, \Delta)\Delta). 
\end{align*}
Hence $ w\cdot \phi(f) = \phi(w \cdot f) $. 

(2) For any $ f \in S_{k} $, we have
\begin{align*}
\deg \phi(f) = \deg \Delta - \deg (f,\Delta) = 
\deg \Delta - (\deg \Delta - \deg f) = \deg f. 
\end{align*}

(3) By Lemma \ref{Stein1}, 
\begin{align*}
f \in \ker\phi \Leftrightarrow ((f,\Delta),\Delta) = 0
\Leftrightarrow (f, \Delta) \in I \cap \mathbb{R}[\partial]\Delta = 0
\Leftrightarrow f \in I. 
\end{align*}

(4) We may assume that $ f $ and $ g $ are homogeneous.
If $ \deg f \neq \deg g $ then $ \langle \phi(f),g \rangle = 0 = \langle f, \phi(g) \rangle $. 
Then we only need to verify the assertion when $ \deg f = \deg g $. 
Put $ F:=(f,\Delta) $ and $ G:=(g,\Delta) $.
We have
\begin{align*}
\langle \phi(f), g \rangle
= \langle g, \phi(f) \rangle
= \langle g, (F,\Delta) \rangle
=\langle gF, \Delta \rangle
=\langle F, (g,\Delta) \rangle
=\langle F,G \rangle.  
\end{align*}
Hence $ \langle f,\phi(g) \rangle = \langle G, F \rangle = \langle F, G \rangle = \langle \phi(f), g \rangle$. 
\end{proof}
\begin{flushleft}
{\it Remark.} 
The map $\phi_{h}(f):=((f,h),h)$ satisfies the properties $(2)$ and $(4)$, 
where $ h $ is a homogeneous polynomial. 
The proofs go similarly to Proposition \ref{Prop:property of phi}.
\end{flushleft}
The map $\phi$ induce a linear map 
$\tilde{\phi}\,:\,\Omega^1(V)^W\longrightarrow\Omega_{W}$ 
defined by $\tilde{\phi}(\sum f\otimes x):=\sum\phi(f)\otimes x$. 
By Proposition \ref{Prop:property of phi} (3),
\begin{align*}
\ker \tilde{\phi} = \left(I\otimes_{\RB} V^{\ast}\right)^W, 
\end{align*}
and then the restriction $ \tilde{\phi}|_{\Omega_{W}}:\Omega_{W}\rightarrow\Omega_{W} $ is an isomorphism by Lemma \ref{Stein1}. 
For any homogeneous element 
$ \omega_{1}=\sum_{j=1}^{n}g_{j}dx_{j},
\omega_{2}=\sum_{j=1}^{n}h_{j}dx_{j} \in \Omega_{W}$ 
with $\deg\omega_{1}=\deg\omega_{2}$, 
by Proposition \ref{Prop:property of phi} (4),
\begin{align*}
\left(\tilde{\phi}(\omega_{1}), \omega_{2}\right)
&=\left(\sum_{j=1}^{n}\phi(g_{j})dx_{j}, \sum_{j=1}^{n}h_{j}dx_{j}\right)
=\sum_{j=1}^{n}\langle\phi(g_{j}), h_{j}\rangle \\
&=\sum_{j=1}^{n}\langle g_{j},\phi(h_{j})\rangle
=\left(\sum_{j=1}^{n}g_{j}dx_{j}, \sum_{j=1}^{n}\phi(h_{j})dx_{j}\right)\\
&=\left(\omega_{1},\tilde{\phi}(\omega_{2})\right). 
\end{align*}
If $\deg\omega_{1}\neq\deg\omega_{2}$, then 
$\left(\tilde{\phi}(\omega_{1}), \omega_{2}\right)=0=
\left(\omega_{1},\tilde{\phi}(\omega_{2})\right).$ 
Hence the map $ \tilde{\phi}|_{\Omega_{W}} $ is symmetric 
with respect to the inner product \eqref{Inner product on Omega_{W}} 
and $ \tilde{\phi}|_{\Omega_{W}} $ is diagonalizable 
with homogeneous eigenvectors. 
\begin{Prop}\label{Proposition:eigenvectors}
There exists a canonical system $\{f_{1},\dots,f_{n}\}$ 
such that $ df_{1},\dots,df_{n} $ are eigenvectors of $ \tilde{\phi} $. 
\end{Prop}
\begin{proof}
Let $U_{d,\lambda}$ be the subspace of $\Omega_{W}$ consisting 
of homogeneous eigenvectors with degree $d$ and eigenvalue $\lambda$. 
We assume that $U_{d_{1},\lambda_{1}}\neq U_{d_{2},\lambda_{2}}$. 
If $d_{1}\neq d_{2}$, then $U_{d_{1},\lambda_{1}}$ is orthogonal to 
$U_{d_{2},\lambda_{2}}$. When $\lambda_{1}\neq \lambda_{2}$, let 
$\omega_{1}\in U_{d_{1},\lambda_{1}}$ and 
$\omega_{2}\in U_{d_{2},\lambda_{2}}$. 
By Proposition \ref{Prop:property of phi} (4), we have 
\begin{align*}
\lambda_{i}\left(\omega_{i},\omega_{j}\right)
= \left(\tilde{\phi}(\omega_{i}),\omega_{j}\right)
= \left(\omega_{i}, \tilde{\phi}(\omega_{j})\right)
= \lambda_{j}\left(\omega_{i},\omega_{j}\right). 
\end{align*}
Then $ (\omega_{i},\omega_{j}) = 0 $. 
This means that $U_{d_{1},\lambda_{1}}$ is orthogonal to
$U_{d_{2},\lambda_{2}}$. 
Hence $\Omega_{W}$ is decomposed into the direct sum of 
orthogonal components $U_{d,\lambda}$. 

Let $\{\omega_{1}\dots,\omega_{n}\}$ be a orthogonal basis for 
$\Omega_{W}=\bigoplus_{d,\lambda}U_{d,\lambda}$. 
The normalization of $ \varepsilon(\omega_{1})\dots,\varepsilon(\omega_{n}) $ 
is the required canonical system.
\end{proof}
\begin{Them}\label{const-can-sys}
For an arbitrary system of basic invariants $\{h_{1},\dots,h_{n}\}$, 
the system 
\begin{align*}
\{\tilde{\phi}(dh_{1}), \dots, \tilde{\phi}(dh_{n})\}
\end{align*}
is a basis for $ \Omega_{W} $. Therefore, 
the Gram-Schmidt orthogonalization of 
\begin{align}
\left\{\varepsilon\circ\tilde{\phi} (dh_{i})
=\sum_{j=1}^n x_{j}\phi(\partial_{j}h_{i})\,\middle|\,
i=1,\dots,n\right\}\label{can-sys-thm}
\end{align}
with respect to the inner product (\ref{inner}) takes 
a canonical system of basic invariants.
\end{Them}
\begin{proof}
We can take a canonical system $ f_{1} \dots,f_{n}$ 
satisfying that $ df_{1}, \dots, df_{n} $ are eigenvectors of 
$ \tilde{\phi} $ by Proposition \ref{Proposition:eigenvectors}. 
Let $ \lambda_{1},\dots,\lambda_{n} $ be the eigenvalues of 
$ df_{1}, \dots, df_{n} $, respectively. 
All of these are nonzero since $ \tilde{\phi}|_{\Omega_{W}} $ 
is an isomorphism. 
We may assume that $ m_{i}:=\deg f_{i} = \deg h_{i} $ with 
$ m_{1} \leq \dots \leq m_{n} $. 
Let $i=1,\dots,n$, and let $r$ be the number of invariants of 
degree $m_{i}$ in $\{h_{1},\dots,h_{n}\}$. 
We may assume $m_{i}=\cdots=m_{i+r-1}$. 
We only need to verify that the linear independence of 
$\{\tilde{\phi}(dh_{i}), \dots, \tilde{\phi}(dh_{i+r-1})\}$. 

For $ j=0,\dots,r-1$, we may write 
\begin{align}
h_{i+j} = \sum_{k=0}^{r-1}a_{j \, k}f_{i+k} + 
P_{j}(f_{1},\dots,f_{i-1}),\label{const-can-sys1}
\end{align}
where $P_{j}=P_{j}(f_{1},\dots,f_{i-1})$ is a polynomial in 
$f_{1}, \dots ,f_{i-1}$ and $a_{j\,k}\in\RB$. 
Since $\{h_{1},\dots,h_{n}\}$ is algebraically independent, 
we have $ \det [a_{j \, k}] \neq 0$. 
Then the $1$-form $dP_{j}$ is in 
$\ker \tilde{\phi}=\left(I\otimes_{\RB} V^{\ast}\right)^W$. 
Applying $\tilde{\phi}\circ d$ 
on the equality \eqref{const-can-sys1}, we have
\begin{align}
\tilde{\phi}(dh_{i+j}) = \sum_{k=0}^{r-1}a_{j \, k}\lambda_{i+k}df_{i+k}. 
\end{align}
Therefore, $ \tilde{\phi}(dh_{i}), \dots, \tilde{\phi}(dh_{i+r-1}) $ 
are linearly independent since $ \det [a_{j \, k}] \neq 0 $. 
\end{proof}
\section{Case observations}

By focusing our mind on the classification of 
finite irreducible reflection groups, 
we notice that the construction (Theorem \ref{const-can-sys}) is more refined. 
For this purpose, we take an orthogonal basis consisting of eigenvectors 
of the linear map $\tilde{\phi}$ for each type of the classification. 
We need to make a consideration of two cases; one of them is the case of 
the types except $D_{n}$ with even $n$ ($n\geq 4$) and 
the other one is the case of type $D_{n}$ with even $n$ ($n\geq 4$). 

If $W$ is not of the type $D_{n}$ with even $n$ ($n\geq 4$), then 
the degrees of basic invariants are distinct. 
If $W$ is of the type $D_{n}$ with even $n$ ($n\geq 4$), then 
the degrees are the numbers $2,4,\dots,n,n,\dots,2n-2$ 
(see \cite[sect. 3.7 Table 1]{Hum}). 
\subsection{Types except $D_{n}$ with even $n$}\label{neqD_2l}

We assume that 
$W$ is not of the type $D_{n}$ with even $n$ ($n\geq 4$). 
We recall the degrees of basic invariants are distinct. 

Therefore, for any basic invariants $ h_{1},\dots,h_{n} $, it follows that 
$\{\tilde{\phi}(dh_{1}), \dots, \tilde{\phi}(dh_{n})\}$ 
is an orthogonal basis for $ \Omega_{W} $. 
Hence the system (\ref{can-sys-thm}) is an orthogonal basis for $ \mathcal{F}$.
\begin{Them}\label{can-sys1}
Let $h_{1},\dots,h_{n}$ be a system of basic invariants 
with $\deg h_{i}=m_{i}$. Then the normalization of a system 
$\left\{\sum_{j=1}^n x_{j}\phi(\partial_{j}h_{i})\mid
i=1,\dots,n\right\}$ 
takes a canonical system of basic invariants. 
\end{Them}
\subsection{Type $D_{n}$ with even $n$ ($n\geq 4$)}
Let $n=2\ell\ (\ell\geq 2)$. In this subsection, we assume $W$ is the 
irreducible finite reflection group of type $D_{n}$. 
In this case, two basic invariants have degree $m_{\ell}=m_{\ell+1}=n$. 
We find a homogeneous polynomial $ f \in \mathcal{F} $ of degree 
$ n $ satisfying that $ df $ is an eigenvector of $ \tilde{\phi} $. 
It is well-known that $ \sum_{j=1}^{n}x_{j}^{2i} \, (i=1,\dots, n-1) $ and 
$\prod_{i=1}^n x_{i}$ form a system of basic invariants. 
Hence $\prod_{i=1}^n x_{i}\in \mathcal{F}$ 
(Iwasaki \cite{Iwa} also constructed a canonical system of basic invariants 
which contains $\prod_{i=1}^n x_{i}$). 
We will prove that $d(\prod_{i=1}^n x_{i})$ is 
an eigenvector of $ \tilde{\phi} $. 

The antiinvariant $\Delta$ is given by 
\begin{align}
\Delta = \prod_{1 \leq i < j \leq n}(x_{i}^{2} - x_{j}^{2}) = 
\sum_{\boldsymbol{a}\in2\mathbb{N}^{n}}c_{\boldsymbol{a}}x^{\boldsymbol{a}}
\end{align}
for some coefficient $ c_{\boldsymbol{a}} \in \mathbb{R}$. 
We denote $|\bm{a}|:=a_{1}+\cdots+a_{n}$ for a multi-index 
$\bm{a}=(a_{1},\dots,a_{n})\in\NB^{n}$. 
We write $\bm{a}\geq \bm{b}$ if $a_{k}\geq b_{k}$ for all $k=1,\dots,n$. 
Put $\bm{1}:=(1,\dots,1)\in\NB^{n}$. 
Let $\bm{e}_{j}\in\NB^{n}$ be the $j$-th unit vector of $\NB^{n}$. 
\begin{Prop}
Let $ f = x_{1}\cdots x_{n} $. 
Then $ df $ is an eigenvector of $ \tilde{\phi} $. 
\end{Prop}
\begin{proof}
Let $ \boldsymbol{a}, \boldsymbol{b}\in 2\mathbb{N}^{n} $ 
be muti-indices with $ |\boldsymbol{a}| = |\boldsymbol{b}| = \deg \Delta $. 
Note that $ \partial_{1}f = x^{\boldsymbol{1}-\boldsymbol{e}_{1}} $. 
Assume that $ ((x^{\boldsymbol{1}-\boldsymbol{e}_{1}}, x^{\boldsymbol{a}}), 
x^{\boldsymbol{b}}) \neq 0 $.
Then $ \boldsymbol{a}-\boldsymbol{1}+\boldsymbol{e}_{1} \geq \boldsymbol{0} $ 
and $\boldsymbol{b}-\boldsymbol{a}+\boldsymbol{1}-
\boldsymbol{e}_{1}\geq \boldsymbol{0}$. 
If $ \boldsymbol{a} \neq \boldsymbol{b} $, the multi-index 
$ \boldsymbol{b} - \boldsymbol{a} $ has a component less than or equal to 
$ -2 $ since $ |\boldsymbol{a}| = |\boldsymbol{b}| $ and 
$ \boldsymbol{a}, \boldsymbol{b} \in 2\mathbb{N}^{n} $. 
Then $ \boldsymbol{b} - \boldsymbol{a} + \boldsymbol{1} - \boldsymbol{e}_{1} 
\not\geq \boldsymbol{0} $. 
This is a contradiction, and $ \boldsymbol{a} = \boldsymbol{b} $. 
Thus $((x^{\boldsymbol{1}-\boldsymbol{e}_{1}}, x^{\boldsymbol{a}}),
 x^{\boldsymbol{b}}) = cx^{\boldsymbol{1} - \boldsymbol{e}_{1}} $
 for some $ c \in \mathbb{R}^{\times} $. 
Therefore, we conclude that
\begin{align}\label{Proposition : eigenvector x1...xn : 1}
((\partial_{1} f, x^{\boldsymbol{a}}), x^{\boldsymbol{b}}) = c\partial_{1}f
\end{align}
for some $ c \in \mathbb{R}^{\times} $. 
By \eqref{Proposition : eigenvector x1...xn : 1}, we have 
\begin{align}\label{Proposition : eigenvector x1...xn : 2}
\phi(\partial_{1}f) = ((\partial_{j}f,\Delta),\Delta) = 
\left(\left(f, \sum_{\boldsymbol{a}\in 2\mathbb{N}^{n}} 
c_{\boldsymbol{a}}x^{\boldsymbol{a}}\right), 
\sum_{\boldsymbol{b}\in 2\mathbb{N}^{n}} c_{\boldsymbol{b}}
x^{\boldsymbol{b}}\right)= \lambda \partial_{1}f
\end{align}
for some $ \lambda \in \mathbb{R} $. 
The map $ \phi $ is a $ W $-homomorphism and $ f $ is invariant. 
Applying the permutation $(1\;j)\in W\ (j=1,\dots,n)$ on the equation 
\eqref{Proposition : eigenvector x1...xn : 2}, we have 
$\phi(\partial_{1}f)= \lambda \partial_{1}f,\dots,
\phi(\partial_{n}f)= \lambda \partial_{n}f$. 
Hence $ \tilde{\phi}(df) = \lambda df $. 
\end{proof}
We take a canonical system $ f_{1}, \dots, f_{n} $ for 
$\mathcal{F}$ which contains the monomial $f_{\ell+1}=\prod_{i=1}^n x_{i}$. 
Note that $\langle f_{\ell+1},f_{\ell+1}\rangle=1$.
Then
\begin{align}
(\tilde{\phi}(df_{\ell}), df_{\ell+1}) = (df_{\ell}, \tilde{\phi}(df_{\ell+1}))
= \lambda(df_{\ell}, df_{\ell+1}) = 0. 
\end{align}
Since $df_{\ell}$ and $ \tilde{\phi}(df_{\ell}) $ are orthogonal to 
$ df_{\ell+1} $, the $ 1 $-form $ df_{\ell} $ is 
an eigenvector of $ \tilde{\phi} $. 
Let $h_{1},\dots,h_{\ell},h_{\ell+1},\dots,h_{n}$ be a 
system of basic invariants. 
We may assume that $\deg h_{i}=m_{i}$ for $i=1,\dots,n$. 
There exist polynomials $P_{\ell},P_{\ell+1}$ 
in $f_{1},\dots,f_{\ell-1}$ such that 
\begin{align*}
h_{\ell}&=a_{1}f_{\ell}+a_{2}f_{\ell+1}+P_{\ell}(f_{1},\dots,f_{\ell-1}),\\
h_{\ell+1}&=a_{3}f_{\ell}+a_{4}f_{\ell+1}
+P_{\ell+1}(f_{1},\dots,f_{\ell-1}),
\end{align*}
where 
\begin{align*}
a_{1}=&\frac{\langle f_{\ell},h_{\ell}\rangle}
{\langle f_{\ell},f_{\ell}\rangle},\ \ \,\quad
a_{2}=\frac{\langle f_{\ell+1},h_{\ell}\rangle}
{\langle f_{\ell+1},f_{\ell+1}\rangle}
=\langle f_{\ell+1},h_{\ell}\rangle,\\
a_{3}=&\frac{\langle f_{\ell},h_{\ell+1}\rangle}
{\langle f_{\ell},f_{\ell}\rangle},\quad
a_{4}=\frac{\langle f_{\ell+1},h_{\ell+1}\rangle}
{\langle f_{\ell+1},f_{\ell+1}\rangle}
=\langle f_{\ell+1},h_{\ell+1}\rangle.
\end{align*}
We have 
$\displaystyle a_{1}a_{4}-a_{2}a_{3}=\det\left[
\begin{matrix}
a_{1}&a_{2}\\
a_{3}&a_{4}
\end{matrix}
\right]\neq 0$ since $h_{1},\dots,h_{n}$ and $f_{1},\dots,f_{n}$ are 
systems of basic invariants. Hence we have the following. 
\begin{Them}\label{can-sys2}
Let $W$ be the irreducible finite reflection group 
of type $D_{n}$ with $n=2\ell$ ($\ell\geq 2$). 
Let $h_{1},\dots,h_{n}$ be a 
system of basic invariants with 
$\deg h_{i}=m_{i}$. Put $f_{\ell+1}=x_{1}\cdots x_{n}$. 
Then the normalization of a system 
\begin{align*}
\sum_{j=1}^{n}x_{j}\phi(\partial_{j}h_{1}),\dots,
\sum_{j=1}^{n}x_{j}\phi\left(
\partial_{j}(bh_{\ell}-ah_{\ell+1})\right),
f_{\ell+1},
\dots,\sum_{j=1}^{n}x_{j}\partial_{j}\phi(\partial_{j}h_{n})
\end{align*}
takes a canonical system of basic invariants, where
\begin{align*}
a=\langle f_{\ell+1},h_{\ell}\rangle
,\ \ \,\quad
b=\langle f_{\ell+1},h_{\ell+1}\rangle.
\end{align*}
\end{Them}
\section*{Acknowledgements}
The authors thank Professor Akihito Wachi
and the referee for their helpful comments.

Norihiro Nakashima \\
Department of Mathematics, \\
Graduate School of Science, \\
Hokkaido University, Sapporo, $060$-$0810$, Japan\\
\vspace{2mm}
\\
Shuhei Tsujie \\
Department of Mathematics, \\
Graduate School of Science, \\
Hokkaido University, Sapporo, $060$-$0810$, Japan\\
\end{document}